\documentclass[11pt,a4paper]{article}
\usepackage[utf8]{inputenc}
\usepackage[T1]{fontenc}
\usepackage{amsmath, amssymb, amsthm}
\usepackage{geometry}
\usepackage{enumitem}
\usepackage{mathrsfs}
\usepackage{bbm}
\usepackage{graphicx}
\usepackage{bm}
\usepackage{enumitem}

\newtheorem{theorem}{Theorem}[section]
\newtheorem{lemma}[theorem]{Lemma}

\theoremstyle{definition}

\newtheorem{remark}{Remark}[section]

\usepackage{tikz,xcolor,hyperref}

\begin{document}


\title{\textbf{Liouville Type Theorem for the Steady Fractional Compressible MHD Equations in $\mathbb{R}^{3}$}\footnote{These authors contributed equally to this work. Authors are listed alphabetically by surname then given name. Authors equally share the first authorship.}}
\author{Zhenyuan Liu\footnote{mx120230298@stu.yzu.edu.cn} \and Weihua Wang \\
School of Mathematical Sciences, Yangzhou University, Yangzhou, China}
\date{}
\maketitle

\begin{flushleft}
\textbf{Correspondence:} Weihua Wang \href{mailto:wangvh@163.com,wangweihua15@mails.ucas.ac.cn}{wangvh@163.com,wangweihua15@mails.ucas.ac.cn}
\end{flushleft}

\noindent \textbf{Received:} 25 May 2026 \quad \textbf{Revised:} xx xxx 2026 \quad \textbf{Accepted:} xx xxx 2026

\vspace{0.5cm}

\section*{ABSTRACT}
This paper is concerned with the Liouville-type problem for the stationary fractional compressible magnetohydrodynamics (MHD) equations.  The main difficulty comes from the nonlocal fractional Laplace operator $(-\Delta)^s$.  To overcome it, we combine the Caffarelli-Silvestre extension technique with  truncation arguments. Under suitable regularity and decay conditions, we prove that the only solution is trivial, establishing a Liouville-type theorem.

\section{Introduction}
In this paper, we consider the Liouville type problem  for the following stationary fractional compressible MHD system in $\mathbb{R}^{3}$:
\begin{equation}\label{eq1.1}\tag{fMHD}
\begin{cases}
    (-\triangle)^{\alpha}u+\text{div}(\rho u\otimes u)-(b \cdot \nabla)b+\nabla p = -\frac{1}{2}\nabla|b|^2 ,\\
    (-\triangle)^{\beta}b+(u \cdot \nabla)b-(b \cdot \nabla)u = 0,\\
     \text{div}(\rho u )=0,\\
     \text{div}~b = 0,
\end{cases}
\end{equation}
where $u, ~b$ and $\rho$ represent the velocity fields, magnetic fields and density, respectively. And
\begin{equation}\label{eq1.2}
  p(\rho)=\eta\rho^\gamma
\end{equation}
is the pressure with constant $\eta>0$ and the adiabatic exponent $\gamma\geq1$. The fractional operator $(-\triangle)^{\alpha},~(-\triangle)^{\beta}$ is defined at the Fourier multiplier by the symbol $|\xi|^{2\alpha}$ and $|\xi|^{2\beta}$, respectively.

Setting $\alpha=1$, $b=0$ and constant $\rho$, the fractional MHD equations \eqref{eq1.1} reduce to the classical 3D stationary Navier-Stokes equations:
    \begin{equation*}
    \begin{cases}
    -\triangle u + (u\cdot\nabla)u+\nabla p = 0, \\
    \nabla\cdot u = 0.
    \end{cases}
    \end{equation*}
    D-solutions (finite Dirichlet integral, $\lim\limits_{|x|\to\infty}u(x)=0$) are standard solutions.
Gilbarg and Weinberger \cite{GW1978} proved all 2D D-solutions trivial. For 3D: Galdi \cite{Gal2011} showed $u\equiv0$ if $u\in L^{\frac{9}{2}}(\mathbb{R}^3)$; Chae and Wolf \cite{DCha2016} improved this via $\int_{\mathbb{R}^3}\frac{|u|^2}{\ln(2+|u|^{-1})}dx<\infty$; Chae \cite{C2014} via $\triangle u\in L^{\frac{5}{6}}(\mathbb{R}^3)$; Seregin \cite{SG2016} for $L^6(\mathbb{R}^3)\cap BMO^{-1}$; Kozono et al. \cite{KH2017} for vorticity $\sim|x|^{-5/3}$ or small $\|u\|_{L^{9/2,\infty}}$; Chamorro et al. \cite{DO2018} extended to $3\leq p\leq\frac{9}{2}$, Yuan and Xiao \cite{YX2020} to $2\leq p\leq\frac{18}{5}$.

Setting $\alpha=1$ and $b=0$, system \eqref{eq1.1} reduces to the compressible Navier-Stokes equations, whose steady Liouville-type theorems have been extensively studied.
Chae \cite{CD2012} showed that for $2\le d\le 6$,
    \[
    \|\rho\|_{L^\infty(\mathbb{R}^d)}+\|\nabla u\|_{L^2(\mathbb{R}^d)}+\|u\|_{L^{\frac{d}{d-1}}(\mathbb{R}^d)}<\infty,
    \]
    and for $d\ge7$,
    \[
    \|\rho\|_{L^\infty(\mathbb{R}^d)}+\|\nabla u\|_{L^2(\mathbb{R}^d)}+\|u\|_{L^{\frac{d}{d-1}}(\mathbb{R}^d)}+\|u\|_{L^{\frac{3d}{d-1}}(\mathbb{R}^d)}<\infty,
    \]
    imply trivial solutions ($u=0$, constant $\rho$). Li and Yu \cite{LY2014} improved this: for $d\ge4$, $\rho\in L^\infty(\mathbb{R}^d)$ and $u\in\dot{H}^1(\mathbb{R}^d)$ suffice; for $d=2,3$, $u\in L^{\frac{3d}{d-1}}(\mathbb{R}^d)$ is additional.

For \(\alpha=\beta=1\) and constant \(\rho\), \eqref{eq1.1} reduces to stationary incompressible MHD equations. Schulz \cite{Sch2019} showed \(u=b\equiv0\) if \((u,b)\in L^p(\mathbb{R}^3)\cap BMO^{-1}(\mathbb{R}^3)\) (\(p\in(2,6]\)). Yuan and Xiao \cite{YX2020} improved this to \(u,b\in L^p(\mathbb{R}^3)\) (\(2\le p\le\frac{9}{2}\)). Fan and Wang \cite{wangf2021} proved \(u=b\equiv0\) if \(u,b\in L_{x_1}^q L_{x_2}^q L_{x_3}^r(\mathbb{R}^3)\) (\(q,r\ge3\), \(\frac{2}{q}+\frac{1}{r}\ge\frac{2}{3}\)) or \(u,b\in L_{x_1}^p L_{x_2}^q L_{x_3}^r(\mathbb{R}^3)\) (\(p,q,r\ge3\), \(\frac{1}{p}+\frac{1}{q}+\frac{1}{r}\ge\frac{2}{3}\)). These generalize isotropic Lebesgue theory via anisotropic integrability. For classical MHD results, see Wang and Guo \cite{WangGuo2024} and references therein.

Setting $\alpha=\beta=1$ in \eqref{eq1.1} yields the stationary compressible MHD equations. Wu \cite{wf2022} showed smooth solutions $(\rho, u, b)$ are trivial ($u=b\equiv 0$, constant $\rho$) if $\rho\in L^\infty(\mathbb{R}^3)$, $\int_{\mathbb{R}^{3}}(|\nabla u|^2+|\nabla b|^2)dx<\infty$, and $u,b$ satisfy one of the following:
\begin{enumerate}[label=(\arabic*)]
\item
$u,b\in L^{p,\infty}_{x_1}L^{q,\infty}_{x_2}L^{r,\infty}_{x_3}$ for $p,q,r\in(3,+\infty]$ with $\frac{1}{p}+\frac{1}{q}+\frac{1}{r}>\frac{2}{3}$,
\item
$u,b\in L^{p,s_1}_{x_1}L^{q,s_2}_{x_2}L^{r,s_3}_{x_3}$ for  $p,q,r\in(3,+\infty]$, $s_i\in(3,+\infty)$ with $\frac{1}{p}+\frac{1}{q}+\frac{1}{r}>\frac{2}{3}$,
\item
$u,b\in L^{p,\infty}_{x_1}L^{q,\infty}_{x_2}L^{r,\infty}_{x_3}$ for $p,q,r\in(\frac{3}{2},+\infty]$ with $\frac{1}{p}+\frac{1}{q}+\frac{1}{r}\geq1$,
\item
$u,b\in L^{p,s_1}_{x_1}L^{q,s_2}_{x_2}L^{r,s_3}_{x_3}$ for $p,q,r\in(\frac{3}{2},+\infty]$, $s_i\in(3,+\infty)$ with $\frac{1}{p}+\frac{1}{q}+\frac{1}{r}\geq1$.
\end{enumerate}
For more Liouville-type theorems for steady compressible fluid equations, see \cite{ZP2020,XA2018,ZP20202,ZY2018} and references therein.

Recently, Zeng~\cite{zengy2025} obtained Liouville-type theorems for \eqref{eq1.1}, proving that any solution $(u,b)\in\dot{H}^\alpha(\mathbb{R}^3)\times\dot{H}^\beta(\mathbb{R}^3)$ with $u=(u_1,u_2,u_3)$, $b=(b_1,b_2,b_3)$ and $(u_j,b_j)\in L^{\vec{p}_j}(\mathbb{R}^3)\times L^{\vec{q}_j}(\mathbb{R}^3)$ is trivial if
\[
\sum_{l=1}^3\frac1{p_{j,l}}\ge\frac23,\quad
\sum_{l=1}^3\frac1{q_{j,l}}\ge\frac23,\quad
p_{j,l},q_{j,l}\in[3,\infty),\quad \forall j,l=1,2,3.
\]

The main difficulty lies in the nonlocality of fractional operators \((-\triangle)^\alpha\) and \((-\triangle)^\beta\). For \(\alpha,\beta\in(0,1)\), we use the Caffarelli-Silvestre extension method to convert these fractional Laplacians into local degenerate elliptic equations in higher dimensions, enabling classical estimation techniques to derive a Liouville-type result.

Under \(u,b\in\dot{H}^\alpha(\mathbb{R}^3)\times\dot{H}^\beta(\mathbb{R}^3)\) and suitable integrability conditions, we prove solution triviality. Via the Sobolev embedding \(\dot{H}^\alpha(\mathbb{R}^3)\hookrightarrow L^{\frac{6}{3-2\alpha}}(\mathbb{R}^3)\) and H"{o}lder inequality, we derive \(\frac{1}{2}\leq\alpha,\beta<1\), consistent with Zeng \cite{zy2025} in anisotropic Lebesgue spaces, leading to the following conclusion.

\begin{theorem}\label{th1.1}
Let $\frac{1}{2}\leq\alpha,~\beta<1$, $\rho\in L^{\infty}(\mathbb{R}^{3})$ and $(u,b) \in \dot{H}^\alpha(\mathbb{R}^3)\times\dot{H}^\beta(\mathbb{R}^3)$ be a smooth solution to system $\eqref{eq1.1}$. If $u$ and $b$ satisfy:
$\begin{cases}
        u \in L^p(\mathbb{R}^3)\cap L^q(\mathbb{R}^3) &\text{with}~~ 2\leq p\leq\frac{9}{2},~1\leq q \leq\frac{3}{2},\\
        b \in L^p(\mathbb{R}^3) &\text{with}~~ 2\leq p\leq\frac{9}{2},
        \end{cases}$
then $u=b\equiv0$ and $\rho$ is constant.
\end{theorem}

\begin{remark}
In the proof, applying H\"{o}lder inequality (Subsection \ref{subsubsec3.1.1}) gives $$\|u\|_{\mathrm{L}^3(B_{\frac{3}{2}R}\backslash B_{\frac{3}{4}R})}\lesssim R^{\frac{3}{2}-\frac{3}{p}}\|u\|_{\mathrm{L}^p(B_{\frac{3}{2}R}\backslash B_{\frac{3}{4}R})},$$ requiring $p\geq3$. Since $\dot{H}^\alpha(\mathbb{R}^3)\hookrightarrow L^{\frac{6}{3-2\alpha}}(\mathbb{R}^3)$, the condition $u\in \dot{H}^\alpha(\mathbb{R}^3)$ implies $p\leq\frac{6}{3-2\alpha}$, hence $3\leq\frac{6}{3-2\alpha}$ or $\alpha\geq\frac{1}{2}$. The Caffarelli-Silvestre extension requires $\alpha <1$, so $\frac{1}{2}\leq\alpha<1$. The same logic applies to $\frac{1}{2}\leq\beta<1$.
\end{remark}

\begin{remark}
Zeng \cite{zy2025} studied Liouville-type results for fractional compressible MHD equations in anisotropic Lebesgue spaces. For $p_{j,1}=p_{j,2}=p_{j,3}$, his obtained range $3\leq p\leq\frac{9}{2}$ is contained in our results.
\end{remark}

The rest of the paper is organized as follows. Section 2 presents some cutoff functions, the Caffarelli-Silvestre extension \cite[Section 2.3]{CaSi2007}, and several auxiliary lemmas, while Section 3 is dedicated to proving Theorem \ref{th1.1}.

\section{Preliminaries}\label{Sec2}
Let $\|(u,b)\|_{X}^p:=\|u\|_{X}^p+\|b\|_{X}^p$, where $\|\cdot\|_X$ denotes the standard norm of the function space $X$. Throughout this paper, $C$ denotes a finite inessential constant (possibly varying line to line), and $B_R(x_0)$ denotes the ball of radius $R$ centered at $x_0$, with $B_R:=B_R(0)$ for $x_0=0$.
Let $\psi_R(x)\in C_c^\infty(\mathbb{R}^3)$ be a cut-off function satisfying
    $$\psi_R(x)=
    \begin{cases}
      1, & x\in B_{\frac{3}{4}R};\\
      0, & x\in B_{\frac{3}{2}R}^c,
    \end{cases}$$
    where $B_{\frac{3}{2}R}^c:=\mathbb{R}^3\setminus B_{\frac{3}{2}R}$ denotes the complement of the ball $B_{\frac{3}{2}R}$.   Then $\|\nabla^s\psi_R(x)\|_{L^\infty(\mathbb{R}^3)}\leq\frac{C}{R^s}$.

In addition, we set $\chi_R(y)$ be a real  smooth cut-off function on $\mathbb{R}_{+}\cup\{0\}$, defined by
\begin{eqnarray*}
  \chi_R(y)=\left\{
  \begin{array}{ll}
    1, & 0\leq y\leq\frac{3}{4}R; \\
    0, & \hbox{$y>\frac{3}{2}R$},
  \end{array}
\right.
\end{eqnarray*}
then $$|\chi_R'(y)|\leq MR^{-1}, ~|\chi_R''(y)|\leq MR^{-2}$$  for some  constant $M$ independent of $y \in \mathbb{R}_{+}\cup\{0\}$.

To overcome technical difficulties of the fractional Laplacian $(-\triangle)^{s}$ (compared with the Laplacian $\triangle$), we introduce the "Caffarelli-Sivestre" extension operators\cite[Section 2.3]{CaSi2007}, denoting by $\bar{\nabla},~\bar{\triangle}$ the differential operators on the upper half-space $\mathbb{R}^{4}_{+}$.

According to \cite[Section 2.3]{CaSi2007} and \cite[Theorem 2.3]{CMLC2020}, for any $\bm{w}\in\dot{H}^{s}(\mathbb{R}^{3})$ with $s\in(0,1)$, there exists a unique extension $\bm{w^*}$ of $\bm{w}$  in the weighted space $H^{1}(\mathbb{R}^{4}_{+},y^{\lambda_s})$ ( $\lambda_s = 1-2s$) satisfying
\begin{equation}\label{Eq2.1}
\begin{cases}
    \bar{\triangle}_{\lambda_s}\bm{w^*}(x,y) := \bar{\triangle}\bm{w^*}+\frac{\lambda_s}{y}\partial_{y}\bm{w^*} = \frac{1}{y^{\lambda_s}}\overline{\text{div}}(y^{\lambda_s}\bar{\nabla}\bm{w^*}) = 0,\\
    \bm{w^*}(x,0)=\bm{w}(x),\\
     (-\triangle)^{s}\bm{w}(x) = -\bm{C}_{s}\lim\limits_{y\rightarrow 0}y^{\lambda_s}\partial_{y}\bm{w^*}(x,y),\\
     \int_{\mathbb{R}^{3}}|(-\triangle)^{\frac{s}{2}}\bm{w}|^2 \text{d}x =  \int_{\mathbb{R}^{3}}|\xi|^{2s}|\hat{\bm{w}}(\xi)|^2 d\xi = \bm{C}_{s}\int_{\mathbb{R}^{4}_{+}}y^{\lambda_s}|\bar{\nabla}\bm{w^*}|^2 \text{d}x\text{d}y,
\end{cases}
\end{equation}
where $\bm{C}_{s}$ is a constant depending only on $s$.

For the integer-order case, $\nabla w = 0$ implies $w$ is constant. Extending this to the fractional order, the following conclusion follows directly from the proof of \cite[Theorem 2.1]{YJ2022}:

\begin{lemma}\label{le2.3}
\begin{equation*}
  \int_{\mathbb{R}^{4}_{+}}y^{\lambda_s}|\bar{\nabla}\bm{w^*}|^2 \text{d}x\text{d}y=0
\end{equation*}
 implies $w^*$ is almost constant everywhere.
\end{lemma}

\begin{remark}
  By \eqref{Eq2.1}, $\|\Lambda^{s}w\|_{L^{2}(\mathbb{R}^{3})}^{2} =\bm{C}_{s}\int_{\mathbb{R}^{4}_{+}}y^{\lambda_s}|\bar{\nabla}\bm{w^*}|^2 \text{d}x\text{d}y$.  Thus,  $\Lambda^{s}w =0$ implies $w$ is almost constant everywhere via Lemma \ref{le2.3}.
\end{remark}

\begin{lemma}[Lemma 2.2 in \cite{WX2018}]\label{le2.5}
  Let $\alpha \in (0, 1)$ and $u^*$ be the $\alpha$-extension of $u \in L^p(\mathbb{R}^3)$ given by \eqref{Eq2.1}, it follows that
  \begin{equation}\label{2.2}
     \left(\int\int_{R^4_+}y^{1-2\alpha}|u^*|^{\frac{(5-2\alpha)p}{3}} \text{d}x\text{d}y\right)^{\frac{3}{(5-2\alpha)p}}\leq C\|u\|_{L^p(\mathbb{R}^3)}. \\
  \end{equation}
\end{lemma}
In view of the Sobolev embedding $\dot{H}^\alpha(\mathbb{R}^3)\hookrightarrow L^{\frac{6}{3-2\alpha}}(\mathbb{R}^3)$, if we let $p=\frac{6}{3-2\alpha}$, we can obtain
  \begin{equation}\label{2.3}
      \left(\int\int_{R^4_+}y^{1-2\alpha}|u^*|^{\frac{2(5-2\alpha)}{3-2\alpha}} \text{d}x\text{d}y\right)^{\frac{3-2\alpha}{2(5-2\alpha)}}\leq C\|u\|_{\dot{H}^\alpha(\mathbb{R}^3)}. \\
  \end{equation}

\begin{lemma}[Gagliardo-Nirenberg inequality,  ~Lemma 2.2, ~\cite{Wu24} or \cite{Nirenberg11}]\label{le2.4}
		Let $1\leq q,s<\infty$ and $m\leq k$.
		Suppose that $j$ and $\vartheta$ satisfy $m\leq j\leq k$, $0\leq \vartheta\leq 1$ and define $p\in [1,+\infty]$ by
		
		\begin{equation*}
			j- \frac{3}{p}=\left(m-\frac{3}{s}\right)\vartheta+\left(k-\frac{3}{q}\right)(1-\vartheta).
		\end{equation*}
		Then the inequality holds:
		\begin{equation*}
			\|\nabla^{j} u\|_{L^{p}}\leq C	\|\nabla^{m}u\|_{L^{s}}^{\vartheta}\|\nabla^{k} u\|_{L^{q}}^{1-\vartheta}, ~ u\in W^{m,s}(\mathbb{R}^{3})\cap W^{k,q}(\mathbb{R}^{3}),
		\end{equation*}
		where constant $C\geq 0$. Here, when $p=\infty$, we require that $0<\vartheta<1$.
\end{lemma}

\begin{lemma}[\cite{VA2019}, Fractional Leibniz rule]\label{le2.2}
Given that $f$ and $g$ are two smooth functions, then we have the following estimate:
  $$\|(-\triangle)^s(fg)\|_{L^p}\leq C\|(-\triangle)^sf\|_{L^{p_0}}\|g\|_{L^{p_1}}+C\|f\|_{L^{q_0}}\|(-\triangle)^s g\|_{L^{q_1}},$$
  provided $\frac{1}{p}=\frac{1}{p_0}+\frac{1}{p_1}=\frac{1}{q_0}+\frac{1}{q_1}$, with $s>0$, $1 < p < +\infty$, and $1 < p_0, p_1, q_0, q_1 \leq
+\infty$.
\end{lemma}

\begin{lemma}\label{le2.6}
Let $\rho \in L^\infty(\mathbb{R}^3)$. There exists a constant $C > 0$ such that
$$\|\rho \ln\rho\|_{L^\infty(\mathbb{R}^3)} \le C \|\rho\|_{L^\infty(\mathbb{R}^3)}^2 + C \|\rho\|_{L^\infty(\mathbb{R}^3)}^{1/2}.$$
\end{lemma}

\begin{proof}
Consider $f(x) = x\ln x$ for $x>0$. One checks that
\[
|f(x)| \le
\begin{cases}
C_0 x^2, & x > 1,\\
C_0 x^{1/2}, & 0 < x \le 1,
\end{cases}
\]
holds with e.g. $C_0 = 2/e$ (since $\max_{x>0} |x\ln x|/x^2 = 1/e$ and $\max_{0<x\le1} |x\ln x|/x^{1/2} = 2/e$). Substituting $x = \rho(y)$ and taking the essential supremum over $y\in\mathbb{R}^3$ gives
\[
\|\rho\ln\rho\|_{L^\infty} \le C_0 \|\rho\|_{L^\infty}^2 + C_0 \|\rho\|_{L^\infty}^{1/2},
\]
which proves the lemma with $C = C_0$.
\end{proof}

\section{The proof of Theorem \ref{th1.1}}

First, multiplying $\eqref{eq1.1}_1$ by $u\psi_R(x)$ and $\eqref{eq1.1}_2$ by $b\psi_R(x)$,
 \begin{eqnarray*}
   \langle \frac{1}{2}\nabla|b|^2\,\text{,}\,
   u\psi_{R}\rangle &=& \langle (-\triangle)^{\alpha}u\,\text{,}\,
   u\psi_{R}\rangle +\langle \text{div}(\rho u\otimes u)\,\text{,}\,u\psi_{R}\rangle-\langle(b \cdot \nabla)b \,\text{,}\,u\psi_{R}\rangle+\langle\nabla p\,\text{,}\,u\psi_{R} \rangle\\
   & &+\langle (-\triangle)^{\beta}u\,\text{,}\,b\psi_{R}\rangle+\langle(u \cdot \nabla)b\,\text{,}\,b\psi_{R}\rangle-\langle(b \cdot \nabla)u \,\text{,}\,b\psi_{R}\rangle.
 \end{eqnarray*}
  Integrating by parts over $\mathbb{R}^{3}$ and taking into account $\eqref{eq1.1}_3$ and $\eqref{eq1.1}_4$, we have
\begin{enumerate}[label=(\Roman*)]
  \item
$\langle \text{div}(\rho u\otimes u)\,\text{,}\,u\psi_{R}\rangle = \int_{\mathbb{R}^{3}} \rho u_i \partial_i u_j u_j\psi_R \, \text{d}x = \frac{1}{2} \int_{\mathbb{R}^{3}} \rho u_i \partial_i u_j^2 \psi_R \, \text{d}x = - \int_{\mathbb{R}^{3}} (u \cdot \nabla \psi_R) \rho\frac{|u|^2}{2} \, \text{d}x$
  \item\label{2}
$\langle(b \cdot \nabla) b\,\text{,}\,\psi_R u \, \rangle= \int_{\mathbb{R}^{3}} b_i \partial_i b_j \psi_R u_j \, \text{d}x ,$
\item
$\langle\nabla p\,\text{,}\,\psi_Ru\rangle~~\text{remains unchanged}$
\item \label{4}
$\langle(u\cdot\nabla)b\,\text{,}\,\psi_Rb\rangle=\int_{\mathbb{R}^{3}}u_i\partial_ib_j\cdot\psi_Rb_j~\text{d}x=\frac{1}{2}\int_{\mathbb{R}^{3}}u_i\partial_ib_j^2\psi_R~\text{d}x=-\frac{1}{2}\int_{\mathbb{R}^{3}}\partial_iu_ib_j^2\psi_R~\text{d}x\\
-\frac{1}{2}\int_{\mathbb{R}^{3}}u_ib_j^2\partial_i\psi_R~\text{d}x
=-\int_{\mathbb{R}^{3}}(\nabla \cdot u)\frac{|b|^2}{2}\psi_R~\text{d}x-\int_{\mathbb{R}^{3}}(u\cdot\nabla\psi_R)\frac{|b|^2}{2}~\text{d}x,$
\item \label{5}
$\langle(b\cdot\nabla)u\,\text{,}\,\psi_Rb\rangle=\int_{\mathbb{R}^{3}}b_i\partial_iu_j\cdot\psi_Rb_j~\text{d}x=-\int_{\mathbb{R}^{3}}b_iu_j\partial_i(\psi_Rb_j)~\text{d}x=-\int_{\mathbb{R}^{3}}b_iu_j\partial_i\psi_Rb_j~\text{d}x\\
-\int_{\mathbb{R}^{3}}b_iu_j\psi_R\partial_ib_j~\text{d}x.$
\item \label{6}
$\langle \frac{1}{2}\nabla|b|^2\,\text{,}\, u\psi_{R}\rangle=-\int_{\mathbb{R}^{3}}\frac{|b|^2}{2}(\nabla\cdot u)\psi_R~\text{d}x-\int_{\mathbb{R}^{3}}\frac{|b|^2}{2}u\cdot\nabla\psi_R$
\end{enumerate}
We notice that there are the same items in \ref{2} and \ref{5}, so we add the two equations together to obtain
$$\ref{2}+ \ref{5}=-\int_{\mathbb{R}^{3}}b_iu_j\partial_i\psi_Rb_j~\text{d}x=-\int_{\mathbb{R}^{3}}(b\cdot\nabla\psi_R)(b\cdot u).$$
Similarly, \ref{4} and \ref{6} have the same item, we can have
$$\ref{4}- \ref{6}=\int_{\mathbb{R}^{3}}\frac{|b|^2}{2}u\cdot\nabla\psi_R-\int_{\mathbb{R}^{3}}(u\cdot\nabla\psi_R)\frac{|b|^2}{2}~\text{d}x.$$
Combining the above two equations yields
\begin{align}\label{eq3.1}
    &\langle(-\triangle)^{\alpha}u\,\text{,}\,\psi_Ru\rangle+\langle(-\triangle)^{\beta}b\,\text{,}\,\psi_Rb\rangle\nonumber\\
    =&\int_{\mathbb{R}^{3}}(u\cdot\nabla\psi_R)\left(\frac{\rho|u|^2}{2}+\frac{|b|^2}{2}\right) \nonumber- \int_{\mathbb{R}^{3}}(b\cdot\nabla\psi_R)(b\cdot u)~\text{d}x\nonumber\\
    &+\int_{\mathbb{R}^{3}}(\nabla \cdot u)\frac{|b|^2}{2}\psi_R~\text{d}x+\int_{\mathbb{R}^{3}}\psi_R u\cdot \nabla p\text{d}x-\int_{\mathbb{R}^{3}}\frac{|b|^2}{2}u\cdot\nabla\psi_R.
\end{align}

In order to estimate $(-\triangle)^\alpha u$ and $(-\triangle)^\beta b$, we multiply \eqref{Eq2.1}$_{1}$ by $C_\alpha y^{\lambda_\alpha}u^*\psi_R(x)\chi_R(y)$ and then integrate to obtain
  \begin{align}\label{eq3.2}
    0=&C_{\alpha}\int_{\mathbb{R}^{4}_+}\overline{\text{div}}(y^{\lambda_\alpha}\bar{\nabla}u^*)\cdot u^*(x,y)~\psi_R(x)\chi_R(y) \text{d}x\text{d}y\nonumber\\
    =&-C_{\alpha}\int_{\mathbb{R}^{3}}\lim\limits_{y\rightarrow 0}~y^{\lambda_\alpha}\partial_{y}u^*\cdot u^*(x,y)\psi_R(x)\chi_R(y) \text{d}x\nonumber\\
    &-C_{\alpha}\int_{\mathbb{R}^{4}_+}y^{\lambda_\alpha}|\bar{\nabla}u^*|^2~\psi_R(x)\chi_R(y)~ \text{d}x\text{d}y - C_{\alpha}\int_{\mathbb{R}^{4}_+}y^{\lambda_\alpha}\bar{\nabla}u^*\cdot u^*\bar{\nabla}(\psi_R(x)\chi_R(y)) \text{d}x\text{d}y\nonumber\\
    =&\langle(-\triangle)^\alpha u\,\text{,}\,\psi_Ru\rangle-C_{\alpha}\int_{\mathbb{R}^{4}_+}y^{\lambda_\alpha}|\bar{\nabla}u^*|^2~\psi_R(x)\chi_R(y) \text{d}x\text{d}y\nonumber\\
& + C_{\alpha}\int_{\mathbb{R}^{4}_+}\frac{|u^*|^2}{2}\bar{\nabla}(y^{\lambda_\alpha}\bar{\nabla}(\psi_R(x)\chi_R(y))) \text{d}x\text{d}y.
\end{align}
Similar to Eq. \eqref{eq3.2}, we have
  \begin{align}\label{eq3.3}
    0=&C_{\beta}\int_{\mathbb{R}^{4}_+}\overline{\text{div}}(y^{\lambda_\beta}\bar{\nabla}b^*)\cdot b^*(x,y)~\psi_R(x)\chi_R(y) \text{d}x\text{d}y\nonumber\\
    =&-C_{\beta}\int_{\mathbb{R}^{3}}\lim\limits_{y\rightarrow 0}~y^{\lambda_\beta}\partial_{y}b^*\cdot b^*(x,y)\psi_R(x)\chi_R(y) \text{d}x\nonumber\\
    &-C_{\beta}\int_{\mathbb{R}^{4}_+}y^{\lambda_\beta}|\bar{\nabla}b^*|^2~\psi_R(x)\chi_R(y) \text{d}x\text{d}y - C_{\beta}\int_{\mathbb{R}^{4}_+}y^{\lambda_\beta}\bar{\nabla}b^*\cdot b^*\bar{\nabla}(\psi_R(x)\chi_R(y)) \text{d}x\text{d}y\nonumber\\
    =&\langle(-\triangle)^\beta b\,\text{,}\,\psi_Rb\rangle- C_{\beta}\int_{\mathbb{R}^{4}_+}y^{\lambda_\beta}|\bar{\nabla}b^*|^2~\psi_R(x)\chi_R(y) \text{d}x\text{d}y\nonumber\\
& + C_{\beta}\int_{\mathbb{R}^{4}_+}\frac{|b^*|^2}{2}\bar{\nabla}(y^{\lambda_\beta}\bar{\nabla}(\psi_R(x)\chi_R(y))) \text{d}x\text{d}y.
\end{align}
Combining the above two equations \eqref{eq3.2} - \eqref{eq3.3}, we can obtain that
  \begin{align*}
0=&\langle(-\triangle)^\alpha u\,\text{,}\,\psi_Ru\rangle+\langle(-\triangle)^\beta b\,\text{,}\,\psi_Rb\rangle\\
&-\bm{C}_{\alpha}\int_{\mathbb{R}^{4}_+}y^{\lambda_\alpha}|\bar{\nabla}u^*|^2~\psi_R(x)\chi_R(y) \text{d}x\text{d}y-\bm{C}_{\beta}\int_{\mathbb{R}^{4}_+}y^{\lambda_\beta}|\bar{\nabla}b^*|^2~\psi_R(x)\chi_R(y) \text{d}x\text{d}y\\
&+ \bm{C}_{\alpha}\int_{\mathbb{R}^{4}_+}\frac{|u^*|^2}{2}\bar{\nabla}(y^{\lambda_\alpha}\bar{\nabla}(\psi_R(x)\chi_R(y))) \text{d}x\text{d}y+\bm{C}_{\beta}\int_{\mathbb{R}^{4}_+}\frac{|b^*|^2}{2}\bar{\nabla}(y^{\lambda_\beta}\bar{\nabla}(\psi_R(x)\chi_R(y))) \text{d}x\text{d}y.
\end{align*}
Since by substituting the equation above into Eq. \eqref{eq3.1}, we have
  \begin{align}\label{eq3.4}
&2\bm{C}_{\alpha}\int_{\mathbb{R}^{4}_+}y^{\lambda_\alpha}|\bar{\nabla}u^*|^2~\psi_R(x)\chi_R(y)~\text{d}x\text{d}y + 2\bm{C}_{\beta}y^{\lambda_\beta}|\bar{\nabla}b^*|^2~\psi_R(x)\chi_R(y)~\text{d}x\text{d}y\nonumber\\
\leq&\int_{\mathbb{R}^{3}}(u\cdot\nabla\psi_R)\rho|u|^2+\int_{\mathbb{R}^{3}}(u\cdot\nabla\psi_R)|b|^2+ 2\int_{\mathbb{R}^{3}}(b\cdot\nabla\psi_R)(b\cdot u)~\text{d}x\nonumber\\
    &+\int_{\mathbb{R}^{3}}\frac{|b|^2}{2}u\cdot\nabla\psi_R+ \bm{C}_{\alpha}\int_{\mathbb{R}^{4}_+}|u^*|^2\bar{\nabla}(y^{\lambda_\alpha}\bar{\nabla}(\psi_R(x)\chi_R(y))) \text{d}x\text{d}y\nonumber\\
&+\bm{C}_{\beta}\int_{\mathbb{R}^{4}_+}|b^*|^2\bar{\nabla}(y^{\lambda_\beta}\bar{\nabla}(\psi_R(x)\chi_R(y))) \text{d}x\text{d}y+2\int_{\mathbb{R}^{3}}\psi_R u\cdot \nabla p\text{d}x\nonumber\\
:=& \sum_{i=1}^{7} I_{i}.
\end{align}
Next, we will divide the proof of Theorem\ref{th1.1} into two parts.

\subsection{Case $3\leq p \leq \frac{9}{2}$}\label{subsubsec3.1.1}
For the first term in Eq. \eqref{eq3.4}, we have
\begin{equation}\label{eq3.18}
 I_1=\int_{\mathbb{R}^{3}}(u\cdot\nabla\psi_{R})\rho|u|^2~\text{d}x\leq \frac{\|\rho\|_{L^\infty(\mathbb{R}^3)}}{R}\|u\|_{\mathrm{L}^3(B_{\frac{3}{2}R}\backslash B_{\frac{3}{4}R})}^3\overset{(p\geq3)}{\lesssim} R^{2-\frac{9}{p}}\|u\|_{\mathrm{L}^p(B_{2R}\backslash B_{\frac{R}{2}})}^3.
\end{equation}
Similarly, we can obtain estimates for $I_2$ and $I_3$:
\begin{align}\label{eq3.19}
    I_2=&\int_{\mathbb{R}^{3}}(u\cdot\nabla\psi_{R})|b|^2~\text{d}x\nonumber\\
    \lesssim&\frac{1}{R}\|u\|_{\mathrm{L}^3(B_{\frac{3}{2}R}\backslash B_{\frac{3}{4}R})}\||b|^2\|_{\mathrm{L}^\frac{3}{2}(B_{\frac{3}{2}R}\backslash B_{\frac{3}{4}R})}\hbox{\qquad(By H\"{o}lder's inequality)}\nonumber\\
    \lesssim&\frac{1}{R}\|u\|_{\mathrm{L}^3(B_{2R}\backslash B_{\frac{R}{2}})}\|b\|_{\mathrm{L}^3(B_{2R}\backslash B_{\frac{R}{2}})}^2\nonumber\\
\lesssim&\frac{1}{R}\|(u,b)\|_{\mathrm{L}^3(B_{2R}\backslash B_{\frac{R}{2}})}^3\nonumber\\
\overset{(p\geq3)}{\lesssim}& R^{2-\frac{9}{p}}\|(u,b)\|_{\mathrm{L}^p(B_{2R}\backslash B_{\frac{R}{2}})}^3
\end{align}
and
\begin{align}\label{eq3.20}
I_3=&2\int_{\mathbb{R}^{3}}(b\cdot u)(b\cdot \nabla\psi_R(x)) \text{d}x\lesssim \int_{\mathbb{R}^{3}}|u||b|^2|\nabla\psi_R|~\text{d}x \nonumber\\
\lesssim&\frac{1}{R}\|(u,b)\|_{\mathrm{L}^3(B_{2R}\backslash B_{\frac{R}{2}})}^3\nonumber\\
\overset{(p\geq3)}{\lesssim}& R^{2-\frac{9}{p}}\|(u,b)\|_{\mathrm{L}^p(B_{2R}\backslash B_{\frac{R}{2}})}^3.
\end{align}
Moreover, we can easily notice that $I_4$ has the same estimate as $I_2$ and $I_3$:
\begin{align}\label{eq3.6}
I_4=&\int_{\mathbb{R}^{3}}\frac{|b|^2}{2}u\cdot\nabla\psi_R~\text{d}x\lesssim\int_{\mathbb{R}^{3}}|u||b|^2|\nabla\psi_R|~\text{d}x\overset{(p\geq3)}{\lesssim} R^{2-\frac{9}{p}}\|(u,b)\|_{\mathrm{L}^p(B_{2R}\backslash B_{\frac{R}{2}})}^3.
\end{align}
For $I_5$, we notice that
\begin{align}\label{3.11}
I_5 =& - C_{\alpha}\int_{\mathbb{R}^{4}_+}y^{\lambda_\alpha}\bar{\nabla}u^*\cdot u^*\bar{\nabla}(\psi_R(x)\chi_R(y)) \text{d}x\text{d}y\nonumber\\
  \lesssim& \int_{\mathbb{R}^{4}_+}y^{\lambda_\alpha}\bar{\nabla}u^*\cdot u^*\nabla\psi_R(x)\chi_R(y) \text{d}x\text{d}y + \int_{\mathbb{R}^{4}_+}y^{\lambda_\alpha}\bar{\nabla}u^*\cdot u^*\psi_R(x)\bar{\nabla}\chi_R(y) \text{d}x\text{d}y\nonumber\\
    \lesssim&\frac{1}{R}\int_{0}^{\frac{3}{2}R}\int_{B_{\frac{3}{2}R}\backslash B_{\frac{3}{4}R}}y^{\lambda_\alpha}\bar{\nabla}u^*\cdot u^*~\text{d}x\text{d}y+ \frac{1}{R}\int_{\frac{3}{4}R}^{\frac{3}{2}R}\int_{B_{\frac{3}{2}R}}y^{\lambda_\alpha}\bar{\nabla}u^*\cdot u^*~\text{d}x\text{d}y.
 \end{align}
 Next, applying H\"{o}lder's inequality and Lemma \ref{le2.5}, we obtain
\begin{align}\label{3.10}
  I_5 \lesssim& \frac{1}{R}\left(\int_{0}^{\frac{3}{2}R}\int_{B_{\frac{3}{2}R}\backslash B_{\frac{3}{4}R}}y^{\lambda_\alpha}|\bar{\nabla}u^*|^2~\text{d}x\text{d}y\right)^\frac{1}{2}\left(\int_{0}^{\frac{3}{2}R}\int_{B_{\frac{3}{2}R}\backslash B_{\frac{3}{4}R}}y^{\lambda_\alpha}|u^*|^{\frac{2(5-2\alpha)}{3-2\alpha}}~\text{d}x\text{d}y\right)^\frac{3-2\alpha}{2(5-2\alpha)}\nonumber\\
  &\left(\int_{0}^{\frac{3}{2}R}\int_{B_{\frac{3}{2}R}\backslash B_{\frac{3}{4}R}}y^{\lambda_\alpha}~\text{d}x\text{d}y\right)^\frac{1}{5-2\alpha}\nonumber\\
  &+\frac{1}{R}\left(\int_{\frac{3}{4}R}^{\frac{3}{2}R}\int_{B_{\frac{3}{2}R}}y^{\lambda_\alpha}|\bar{\nabla}u^*|^2~\text{d}x\text{d}y\right)^\frac{1}{2}\left(\int_{\frac{3}{4}R}^{\frac{3}{2}R}\int_{B_{\frac{3}{2}R}}y^{\lambda_\alpha}|u^*|^{\frac{2(5-2\alpha)}{3-2\alpha}}~\text{d}x\text{d}y\right)^\frac{3-2\alpha}{2(5-2\alpha)}\nonumber\\
  &\left(\int_{\frac{3}{4}R}^{\frac{3}{2}R}\int_{B_{\frac{3}{2}R}}y^{\lambda_\alpha}~\text{d}x\text{d}y\right)^\frac{1}{5-2\alpha}\hbox{\qquad(By H\"{o}lder's inequality)}\nonumber\\
\lesssim &\|u\|_{\dot{H}^\alpha(\mathbb{R}^3)}\left(\int_{0}^{\frac{3}{2}R}\int_{B_{\frac{3}{2}R}\backslash B_{\frac{3}{4}R}}y^{\lambda_\alpha}|\bar{\nabla}u^*|^2~\text{d}x\text{d}y\right)^\frac{1}{2}\nonumber\\
  &+\|u\|_{\dot{H}^\alpha(\mathbb{R}^3)}\left(\int_{\frac{3}{4}R}^{\frac{3}{2}R}\int_{B_{\frac{3}{2}R}}y^{\lambda_\alpha}|\bar{\nabla}u^*|^2~\text{d}x\text{d}y\right)^\frac{1}{2}. \hbox{\qquad(By Lemma \ref{le2.5})}
\end{align}

Then we only need to repeat the procedure in \eqref{3.11}~-~\eqref{3.19} by replacing $u$ with $b$ and $\alpha$ with $\beta$, we can have the estimation of $I_6$
\begin{equation}\label{3.16}
  I_6 \lesssim \|b\|_{\dot{H}^\beta(\mathbb{R}^3)}\left(\int_{0}^{\frac{3}{2}R}\int_{B_{\frac{3}{2}R}\backslash B_{\frac{3}{4}R}}y^{\lambda_\beta}|\bar{\nabla}b^*|^2~\text{d}x\text{d}y\right)^\frac{1}{2}
  +\|b\|_{\dot{H}^\beta(\mathbb{R}^3)}\left(\int_{\frac{3}{4}R}^{\frac{3}{2}R}\int_{B_{\frac{3}{2}R}}y^{\lambda_\beta}|\bar{\nabla}b^*|^2~\text{d}x\text{d}y\right)^\frac{1}{2}.
\end{equation}

Finally, in order to estimate $I_7$, we divide the problem into two cases: $\gamma > 1$ and $\gamma  = 1$.

Case1 : $\gamma>1$. We write the pressure term in the following form:
$$\nabla p=\eta\nabla p ^\gamma=\frac{\eta\gamma}{\gamma-1}\rho\nabla\rho^{\gamma-1}.$$
Then, we have
\begin{align*}
  I_7 =& 2\int_{\mathbb{R}^{3}}\psi_R u\cdot \nabla p\text{d}x \\
 \leq& \frac{C\gamma}{\gamma-1}\int_{\mathbb{R}^{3}}\psi_R \rho u\cdot \nabla\rho^{\gamma-1} \text{d}x\\
 =&-\frac{C\gamma}{\gamma-1}\int_{\mathbb{R}^{3}}\psi_R \text{div}(\rho u)\cdot\rho^{\gamma-1} \text{d}x-\frac{C\gamma}{\gamma-1}\int_{\mathbb{R}^{3}}\rho^{\gamma}u\cdot \nabla\psi_R\text{d}x\\
 \leq&\frac{C\gamma}{\gamma-1}\int_{\mathbb{R}^{3}}\rho^{\gamma}u\cdot \nabla\psi_R. \hbox{\qquad\qquad(by~$\eqref{eq1.1}_3$)}\\
\end{align*}
Thus we can obtain
\begin{align}\label{3.17}
  |I_7| \leq &\frac{C\gamma}{\gamma-1}\int_{\mathbb{R}^{3}}\rho^{\gamma}|u| |\nabla\psi_R|\text{d}x\leq \frac{C}{R}\int_{B_{\frac{3}{2}R}\backslash B_{\frac{3}{4}R}}\rho^{\gamma}|u|\text{d}x \nonumber\\
  \leq& CR^{2-\frac{3}{q}}\|\rho\|_{L^\infty(\mathbb{R}^3)}^{\gamma}\|u\|_{L^q(B_{\frac{3}{2}R}\backslash B_{\frac{3}{4}R})}.
\end{align}

Case 2: $\gamma=1$. In this case, $P = a\rho$, hence we have
\begin{align*}
\nabla p=a\nabla\rho=a\rho\nabla\ln\rho.
\end{align*}
According to $\eqref{eq1.1}_3$, we use $\text{div}(\rho u)=0$ and apply Lemma \ref{le2.6} to obtain
\begin{align}\label{3.15}
  I_7= & a\int_{\mathbb{R}^{3}}\psi_R\rho u\cdot\nabla\ln\rho~\text{d}x \nonumber\\
  \lesssim & \int_{\mathbb{R}^{3}}\psi_R\text{div}(\rho u)\ln\rho~\text{d}x+\int_{\mathbb{R}^{3}}(\rho\ln\rho)u\cdot\nabla\psi_R~ \text{d}x \nonumber\\
\lesssim &\int_{\mathbb{R}^{3}}(\rho\ln\rho)u\cdot\nabla\psi_R~ \text{d}x\nonumber\\
\lesssim&\frac{1}{R}\|\rho\ln\rho\|_{L^\infty(\mathbb{R}^3)}\|u\|_{L^q(B_{\frac{3}{2}R}\backslash B_{\frac{3}{4}R})}\|1\|_{L^{\frac{q}{q-1}}(\mathbb{R}^3)}\hbox{\qquad(By H\"{o}lder's inequality)}\nonumber\\
\lesssim&R^{2-\frac{3}{q}}\|\rho\ln\rho\|_{L^\infty(\mathbb{R}^3)}\|u\|_{L^q(B_{\frac{3}{2}R}\backslash B_{\frac{3}{4}R})}\nonumber\\
\lesssim&R^{2-\frac{3}{q}}(\|\rho\|_{L^\infty(\mathbb{R}^3)}^2+\|\rho\|_{L^\infty(\mathbb{R}^3)}^{\frac{1}{2}})\|u\|_{L^q(B_{\frac{3}{2}R}\backslash B_{\frac{3}{4}R})}.\hbox{\qquad(By Lemma \ref{le2.6})}
\end{align}
It is easy to observe that, regardless of whether $\gamma > 1$ or $\gamma = 1$, when $2-\frac{3}{q}<0$, we can obtain $I_7\rightarrow 0$ as $R\rightarrow \infty$.

So, by summing \eqref{eq3.18}, \eqref{eq3.19}, \eqref{eq3.20}, \eqref{eq3.6}, \eqref{3.10}, \eqref{3.16} and \eqref{3.17} or \eqref{3.15}, we conclude that
  \begin{align}\label{eq3.13}
&2\bm{C}_{\alpha}\int_{\mathbb{R}^{4}_+}y^{\lambda_\alpha}|\bar{\nabla}u^*|^2~\psi_R(x)\chi_R(y)~\text{d}x\text{d}y + 2\bm{C}_{\beta}y^{\lambda_\beta}|\bar{\nabla}b^*|^2~\psi_R(x)\chi_R(y)~\text{d}x\text{d}y\nonumber\\
\leq&\sum_{i=1}^{6}I_i\nonumber\\
\lesssim&\underbrace{R^{2-\frac{9}{p}}\|(u,b)\|_{\mathrm{L}^p(B_{2R}\backslash B_{\frac{R}{2}})}^3}_{I_{1} + I_{2}+ I_{3} + I_{4}\lesssim} +I_7 \nonumber\\
&+  \underbrace{\|u\|_{\dot{H}^\alpha(\mathbb{R}^3)}\left(\int_{0}^{\frac{3}{2}R}\int_{B_{\frac{3}{2}R}\backslash B_{\frac{3}{4}R}}y^{\lambda_\alpha}|\bar{\nabla}u^*|^2~\text{d}x\text{d}y\right)^\frac{1}{2}+\|u\|_{\dot{H}^\alpha(\mathbb{R}^3)}\left(\int_{\frac{3}{4}R}^{\frac{3}{2}R}\int_{B_{\frac{3}{2}R}}y^{\lambda_\alpha}|\bar{\nabla}u^*|^2~\text{d}x\text{d}y\right)^\frac{1}{2}}_{I_{5}\lesssim}\nonumber\\
    &+ \underbrace{ \|b\|_{\dot{H}^\beta(\mathbb{R}^3)}\left(\int_{0}^{\frac{3}{2}R}\int_{B_{\frac{3}{2}R}\backslash B_{\frac{3}{4}R}}y^{\lambda_\beta}|\bar{\nabla}b^*|^2~\text{d}x\text{d}y\right)^\frac{1}{2}
  +\|b\|_{\dot{H}^\beta(\mathbb{R}^3)}\left(\int_{\frac{3}{4}R}^{\frac{3}{2}R}\int_{B_{\frac{3}{2}R}}y^{\lambda_\beta}|\bar{\nabla}b^*|^2~\text{d}x\text{d}y\right)^\frac{1}{2}}_{I_{6}\lesssim}.
  \end{align}
 In addition, since $u\in\dot{H}^\alpha(\mathbb{R}^3),~b\in\dot{H}^\beta(\mathbb{R}^3)$, we have
\begin{align}\label{3.19}
&\int_{0}^{2R}\int_{B_{2R}\backslash B_{\frac{R}{2}}}y^{\lambda_\alpha}|\bar{\nabla}u^*|^2~\text{d}x\text{d}y\rightarrow 0,~~\int_{\frac{3}{4}R}^{\frac{3}{2}R}\int_{B_{\frac{3}{2}R}}y^{\lambda_\alpha}|\bar{\nabla}u^*|^2 \text{d}x\text{d}y\rightarrow 0\nonumber\\
\text{ and }~~~~~& \nonumber\\
&\int_{0}^{2R}\int_{B_{2R}\backslash B_{\frac{R}{2}}}y^{\lambda_\beta}|\bar{\nabla}b^*|^2~\text{d}x\text{d}y\rightarrow0,
~~\int_{\frac{3}{4}R}^{\frac{3}{2}R}\int_{B_{\frac{3}{2}R}}y^{\lambda_\beta}|\bar{\nabla}b^*|^2 \text{d}x\text{d}y\rightarrow0,
\end{align}
as $R\rightarrow\infty$. Then it is obvious that $I_5\rightarrow 0$ and $I_6\rightarrow 0$ as $R\rightarrow\infty$.\\
Therefore, RHS of \eqref{eq3.13} converges to zero as $R\rightarrow\infty$ provided
\begin{equation}\nonumber
  \left\{
  \begin{array}{ll}
       2-\frac{9}{p}\leq 0  &  \Leftrightarrow p \leq \frac{9}{2}, \\
      2-\frac{3}{q}\leq 0  &  \Leftrightarrow q \leq \frac{3}{2}, \\
  \end{array}
\right.
\end{equation}
that is, $3\leq p \leq \frac{9}{2}$ and $1\leq q\leq\frac{3}{2}$.

Finally, we can obtain
$$2\bm{C}_{\alpha}\int_{\mathbb{R}^{4}_+}y^{\lambda_\alpha}|\bar{\nabla}u^*|^2~\text{d}x\text{d}y + 2\bm{C}_{\beta}\int_{\mathbb{R}^{4}_+}y^{\lambda_\beta}|\bar{\nabla}b^*|^2~\text{d}x\text{d}y\leq0,$$ i.e., $\bm{C}_{\alpha}\int_{\mathbb{R}^{4}_+}y^{\lambda_\alpha}|\bar{\nabla}u^*|^2~\text{d}x\text{d}y = 0$ and $\bm{C}_{\beta}\int_{\mathbb{R}^{4}_+}y^{\lambda_\beta}|\bar{\nabla}b^*|^2~\text{d}x\text{d}y = 0$,
which implies that $u$ and $b$ are almost constant everywhere according to Lemma \ref{le2.3} and \eqref{Eq2.1}$_{4}$. Since $(u,b)\in\mathrm{L}^p(\mathbb{R}^3)$ and $(u,b)$ is a smooth solution of \eqref{eq1.1}, we have $u=b=0$.

\subsection{Case $2\leq p \leq3$}\label{subsec3.2}

Firstly,  let $\phi_R(x)\in C_c^\infty(\mathbb{R}^3)$ be a non-negative cut-off function (for example, ref. \cite{YX2020}) as following:
\begin{eqnarray*}
  \phi_R(x)=\left\{
  \begin{array}{ll}
    1, & x\in B_{\frac{3}{2}R}\backslash B_{\frac{3}{4}R}, \\
    0, & x\in B_{2R}^c\cup B_{\frac{R}{2}}.
  \end{array}
\right.
\end{eqnarray*}
 And it is a simple matter to $$\|\nabla^s\phi_R\|_{L^{\infty}}\leq\frac{C}{R^s}.$$

 We notice that the estimation of $I_5$, $I_6$ and $I_7$ are the same as Eq. \eqref{3.10}, Eq. \eqref{3.16} and \eqref{3.17} in Subsection \ref{subsubsec3.1.1}, hence we only  need to estimate the nonlinear terms $I_1, ~I_2, ~I_3 \text{ and } I_4$.\\
In order to obtain the estimation of $I_{1}$, we have to estimate $\|u\|_{\mathrm{L}^3(B_{\frac{3}{2}R}\backslash B_{\frac{3}{4}R})}$ as follows
\begin{align}\label{3.18}
   &\|u\|_{\mathrm{L}^3(B_{\frac{3}{2}R}\backslash B_{\frac{3}{4}R})}\leq \|u\phi_R\|_{\mathrm{L}^3(\mathbb{R}^3)}\nonumber\\
\leq& C\|u\phi_R\|_{\mathrm{L}^2(\mathbb{R}^3)}^{1-\frac{1}{2\alpha}}\|\Lambda^\alpha(u\phi_R)\|_{\mathrm{L}^2(\mathbb{R}^3)}^{\frac{1}{2\alpha}}\hbox{\quad (by Gagliardo-Nirenberg inequality)}\nonumber\\
    \leq&C\|u\|_{\mathrm{L}^2(B_{2R}\backslash B_{\frac{R}{2}})}^{1-\frac{1}{2\alpha}}\left(\|\Lambda^\alpha u\|_{\mathrm{L}^2(B_{2R}\backslash B_{\frac{R}{2}})}^{\frac{1}{2\alpha}}\|\phi_R\|_{\mathrm{L}^\infty(\mathbb{R}^3)}^{\frac{1}{2\alpha}}+
    \|u\|_{\mathrm{L}^2(B_{2R}\backslash B_{\frac{R}{2}})}^{\frac{1}{2\alpha}}\|\Lambda^\alpha\phi_R\|_{\mathrm{L}^\infty(\mathbb{R}^3)}^{\frac{1}{2\alpha}}\right)\nonumber\\
    \leq&C\|u\|_{\mathrm{L}^2(B_{2R}\backslash B_{\frac{R}{2}})}^{1-\frac{1}{2\alpha}}\left(\|\Lambda^\alpha u\|_{\mathrm{L}^2(B_{2R}\backslash B_{\frac{R}{2}})}^{\frac{1}{2\alpha}}+\frac{1}{R^\frac{1}{2}}\|u\|_{\mathrm{L}^2(B_{2R}\backslash B_{\frac{R}{2}})}^{\frac{1}{2\alpha}}\right)\nonumber\\
    \leq&C\|u\|_{\mathrm{L}^2(B_{2R}\backslash B_{\frac{R}{2}})}^{1-\frac{1}{2\alpha}}\|\Lambda^\alpha u\|_{\mathrm{L}^2(B_{2R}\backslash B_{\frac{R}{2}})}^{\frac{1}{2\alpha}}+\frac{C}{R^\frac{1}{2}}\|u\|_{\mathrm{L}^2(B_{2R}\backslash B_{\frac{R}{2}})}.
\end{align}
Similar considerations apply to $\|b\|_{\mathrm{L}^3(B_{\frac{3}{2}R}\backslash B_{\frac{3}{4}R})}$
\begin{equation}\label{3.31}
 \|b\|_{\mathrm{L}^3(B_{\frac{3}{2}R}\backslash B_{\frac{3}{4}R})}\leq C\|b\|_{\mathrm{L}^2(B_{2R}\backslash B_{\frac{R}{2}})}^{1-\frac{1}{2\beta}}\|\Lambda^\beta b\|_{\mathrm{L}^2(B_{2R}\backslash B_{\frac{R}{2}})}^{\frac{1}{2\beta}}+\frac{C}{R^\frac{1}{2}}\|b\|_{\mathrm{L}^2(B_{2R}\backslash B_{\frac{R}{2}})}.
\end{equation}
 Thus, with the help of the H\"{o}lder's inequality and  Eq. \eqref{3.18}, we have
    \begin{align}\label{eq3.22}
I_1=&\int_{\mathbb{R}^{3}}(u\cdot\nabla\psi_{R})\rho|u|^2~\text{d}x\leq\frac{C}{R}\|\rho\|_{L^\infty(\mathbb{R}^3)}\|u\|_{\mathrm{L}^3(B_{\frac{3}{2}R}\backslash B_{\frac{3}{4}R})}^3\nonumber\\
\leq&\frac{C}{R}\|u\|_{\mathrm{L}^2(B_{2R}\backslash B_{\frac{R}{2}})}^{3-\frac{3}{2\alpha}}\|\Lambda^\alpha u\|_{\mathrm{L}^2(B_{2R}\backslash B_{\frac{R}{2}})}^{\frac{3}{2\alpha}}+\frac{C}{R^{\frac{5}{2}}}\|u\|_{\mathrm{L}^2(B_{2R}\backslash B_{\frac{R}{2}})}^3\nonumber\\
\overset{p\geq2}{\lesssim}&R^{\frac{3(1-\frac{2}{p})(6\alpha-3)}{4\alpha}-1}\|u\|_{\mathrm{L}^p(B_{2R}\backslash B_{\frac{R}{2}})}^{3-\frac{3}{2\alpha}}\|\Lambda^\alpha u\|_{\mathrm{L}^2(B_{2R}\backslash B_{\frac{R}{2}})}^{\frac{3}{2\alpha}}+ R^{2-\frac{9}{p}}\|u\|_{\mathrm{L}^p(B_{2R}\backslash B_{\frac{R}{2}})}^3.
    \end{align}
According to \eqref{eq3.22}, it is necessary to choose
\begin{equation}\label{e1}
  \left\{
     \begin{array}{ll}
     \frac{3(1-\frac{2}{p})(6\alpha-3)}{4\alpha}-1\leq 0 , \\
     2-\frac{9}{p}\leq 0 . \\
     \end{array}
   \right.
\end{equation}
Hence, we will obtain $I_1 \rightarrow 0$ as $R \rightarrow \infty$.

Then we apply Eq. \eqref{3.18} and Eq. \eqref{3.31} to estimate $I_2$
\begin{align}\label{eq3.23}
  I_2=&\int_{\mathbb{R}^{3}}(u\cdot\nabla\psi_{R})|b|^2~\text{d}x\leq\frac{C}{R}\|u\|_{\mathrm{L}^3(B_{\frac{3}{2}R}\backslash B_{\frac{3}{4}R})} \|b\|_{\mathrm{L}^3(B_{\frac{3}{2}R}\backslash B_{\frac{3}{4}R})}^2 \nonumber\\
  \leq&\frac{C}{R}\left(\|u\|_{\mathrm{L}^2(B_{2R}\backslash B_{\frac{R}{2}})}^{1-\frac{1}{2\alpha}}\|\Lambda^\alpha u\|_{\mathrm{L}^2(B_{2R}\backslash B_{\frac{R}{2}})}^{\frac{1}{2\alpha}}+\frac{C}{R^\frac{1}{2}}\|u\|_{\mathrm{L}^2(B_{2R}\backslash B_{\frac{R}{2}})}\right)\nonumber\\
  &~~\qquad\left(\|b\|_{\mathrm{L}^2(B_{2R}\backslash B_{\frac{R}{2}})}^{2-\frac{1}{\beta}}\|\Lambda^\beta b\|_{\mathrm{L}^2(B_{2R}\backslash B_{\frac{R}{2}})}^{\frac{1}{\beta}}+\frac{C}{R}\|b\|_{\mathrm{L}^2(B_{2R}\backslash B_{\frac{R}{2}})}^2\right)\nonumber\\
  \leq&\frac{C}{R}\left(\|u\|_{\mathrm{L}^2(B_{2R}\backslash B_{\frac{R}{2}})}^{2-\frac{1}{\alpha}}\|\Lambda^\alpha u\|_{\mathrm{L}^2(B_{2R}\backslash B_{\frac{R}{2}})}^{\frac{1}{\alpha}}+\|b\|_{\mathrm{L}^2(B_{2R}\backslash B_{\frac{R}{2}})}^{4-\frac{2}{\beta}}\|\Lambda^\beta b\|_{\mathrm{L}^2(B_{2R}\backslash B_{\frac{R}{2}})}^{\frac{2}{\beta}}\right)\nonumber\\
  &+\frac{C}{R^2}\|u\|_{\mathrm{L}^2(B_{2R}\backslash B_{\frac{R}{2}})}^{2-\frac{1}{\alpha}}\|\Lambda^\alpha u\|_{\mathrm{L}^2(B_{2R}\backslash B_{\frac{R}{2}})}^{\frac{1}{\alpha}}+\frac{C}{R^2}\|b\|_{\mathrm{L}^2(B_{2R}\backslash B_{\frac{R}{2}})}^4\nonumber\\
  &+\frac{C}{R^{\frac{3}{2}}}\|(u,b)\|_{\mathrm{L}^2(B_{2R}\backslash B_{\frac{R}{2}})}^{3-\frac{1}{\beta}}\|\Lambda^\beta b\|_{\mathrm{L}^2(B_{2R}\backslash B_{\frac{R}{2}})}^{\frac{1}{\beta}}+\frac{C}{R^{\frac{5}{2}}}\|(u,b)\|_{\mathrm{L}^2(B_{2R}\backslash B_{\frac{R}{2}})}^{3}.\nonumber\\
  &\hbox{\qquad\qquad\qquad\qquad\qquad\qquad\qquad\qquad\qquad\qquad\qquad(by Young's inequality)}\nonumber\\
  \overset{(p\geq2)}{\lesssim}&R^{\frac{3(1-\frac{2}{p})(2\alpha-1)-2\alpha}{2\alpha}}\|u\|_{\mathrm{L}^p(B_{2R}\backslash B_{\frac{R}{2}})}^{2-\frac{1}{\alpha}}\|\Lambda^\alpha u\|_{\mathrm{L}^2(B_{2R}\backslash B_{\frac{R}{2}})}^{\frac{1}{\alpha}}\nonumber\\
&+ R^{\frac{3(1-\frac{2}{p})(2\beta-1)-\beta}{\beta}}\|b\|_{\mathrm{L}^p(B_{2R}\backslash B_{\frac{R}{2}})}^{4-\frac{2}{\beta}}\|\Lambda^\beta b\|_{\mathrm{L}^2(B_{2R}\backslash B_{\frac{R}{2}})}^{\frac{2}{\beta}}\nonumber\\
&+ R^{\frac{3(1-\frac{2}{p})(2\alpha-1)-4\alpha}{2\alpha}}\|u\|_{\mathrm{L}^p(B_{2R}\backslash B_{\frac{R}{2}})}^{2-\frac{1}{\alpha}}\|\Lambda^\alpha u\|_{\mathrm{L}^2(B_{2R}\backslash B_{\frac{R}{2}})}^{\frac{1}{\alpha}}\nonumber\\
&+ R^{4-\frac{12}{p}}\|b\|_{\mathrm{L}^p(B_{2R}\backslash B_{\frac{R}{2}})}^4+ R^{\frac{3(1-\frac{2}{p})(3\beta-1)-3\beta}{2\beta}}\|(u,b)\|_{\mathrm{L}^p(B_{2R}\backslash B_{\frac{R}{2}})}^{3-\frac{1}{\beta}}\|\Lambda^\beta b\|_{\mathrm{L}^2(B_{2R}\backslash B_{\frac{R}{2}})}^{\frac{1}{\beta}}\nonumber\\
&+ R^{\frac{9(1-\frac{2}{p})-5}{2}}\|(u,b)\|_{\mathrm{L}^p(B_{2R}\backslash B_{\frac{R}{2}})}^{3}.
\end{align}
Furthermore, we can easily observe that:
\begin{align*}
I_3=2\int_{\mathbb{R}^{3}}(b\cdot u)(b\cdot \nabla\psi_R(x)) \text{d}x\leq \frac{C}{R}\|u\|_{\mathrm{L}^3(B_{\frac{3}{2}R}\backslash B_{\frac{3}{4}R})} \|b\|_{\mathrm{L}^3(B_{\frac{3}{2}R}\backslash B_{\frac{3}{4}R})}^2
\end{align*}
and
\begin{align*}
I_4=\int_{\mathbb{R}^{3}}\frac{|b|^2}{2}u\nabla\psi_R~\text{d}x\leq \frac{C}{R}\|u\|_{\mathrm{L}^3(B_{\frac{3}{2}R}\backslash B_{\frac{3}{4}R})} \|b\|_{\mathrm{L}^3(B_{\frac{3}{2}R}\backslash B_{\frac{3}{4}R})}^2.
\end{align*}
Therefore, we only need to use the estimation process of $I_{2}$ to estimate $I_3$ and $I_4$.
Hence, we need to take
\begin{equation}\label{e2}
  \left\{
     \begin{array}{ll}
     \frac{3(1-\frac{2}{p})(2\alpha-1)-2\alpha}{2\alpha}\leq 0 , \\
     \frac{3(1-\frac{2}{p})(2\beta-1)-\beta}{\beta}\leq 0 , \\
     \frac{3(1-\frac{2}{p})(2\alpha-1)-4\alpha}{2\alpha}\leq 0 ,\\
     4-\frac{12}{p}\leq 0 ,\\
     \frac{3(1-\frac{2}{p})(3\beta-1)-3\beta}{2\beta}\leq 0 ,\\
     \frac{9(1-\frac{2}{p})-5}{2}\leq 0.
     \end{array}
   \right.
\end{equation}
Thus, $I_2\rightarrow 0$, $I_3\rightarrow 0$ and $I_4\rightarrow 0$ hold true when $R\rightarrow \infty$.

Finally, combining the vanishing conditions for $I_1-I_7$, i.e. \eqref{e1}, \eqref{e2} and \eqref{3.19} with $1\leq q\leq\frac{3}{2}$,  the exponent $p$ must satisfy the following constraint condition:
\begin{equation}\label{3.35}
  \left\{
     \begin{array}{ll}
     2-\frac{9}{p}\leq 0   &  \hbox{$\Leftrightarrow p \leq \frac{9}{2}$,}\\
     4-\frac{12}{p}\leq 0  &  \hbox{$\Leftrightarrow p\leq 3$,}\\
     \frac{3(1-\frac{2}{p})(3\beta-1)-3\beta}{2\beta}\leq 0  &  \hbox{$\Leftrightarrow p\leq \frac{2(3\beta-1)}{2\beta-1}$,}\\
     \left.
\begin{aligned}
\frac{3\left(1-\frac{2}{p}\right)(2\alpha-1)-4\alpha}{2\alpha} &\leq 0 \\
\frac{3\left(1-\frac{2}{p}\right)(2\alpha-1)-2\alpha}{2\alpha} &\leq 0
\end{aligned}
\right\}  &  \hbox{$\Leftrightarrow 2 - \frac{6}{p} + 3(\frac{1}{p} - \frac{1}{2})\frac{1}{\alpha}\leq 0$}\\
     \end{array}
   \right.
\end{equation}
and
\begin{equation}\label{3.36}
     \begin{cases}
      \frac{3(1-\frac{2}{p})(6\alpha-3)}{4\alpha}-1 \leq 0,\\
      \frac{3(1-\frac{2}{p})(2\beta-1)-\beta}{\beta} \leq 0.
 \end{cases}
\end{equation}
For constraint conditions \eqref{3.35}, note that $\frac{2(3\beta-1)}{2\beta-1} > 3$ when $\alpha,~\beta>\frac{1}{2}$, so \eqref{3.35} simplifies to $p \leq 3$.

Next, we simplify constraint conditions \eqref{3.36}, which is divided into four cases based on the ranges of $\alpha$ and $\beta$:
\begin{enumerate}[label=(\roman*)]
  \item\label{i} For $\frac{1}{2}\leq\alpha\leq\frac{9}{14}$ and $\frac{1}{2}\leq\beta\leq\frac{3}{5}$, we note that
  $$\frac{3(1-\frac{2}{p})(6\alpha-3)}{4\alpha}-1 \leq 0 \text{~~and~~}  \frac{3(1-\frac{2}{p})(2\beta-1)-\beta}{\beta} \leq 0 \text{~always hold true.}$$
Hence, combining with Eq. \eqref{3.35}, we obtain $2\leq p \leq3$. The following cases can be proved similarly to Case \ref{i}, so details are omitted.
  \item For $\frac{1}{2}\leq\alpha\leq\frac{9}{14}$ and $\beta>\frac{3}{5}$,   $ \frac{3(1-\frac{2}{p})(2\beta-1)-\beta}{\beta}\leq0$ implies $ p \leq \frac{6(2\beta-1)}{5\beta-3}$. However, $\frac{6(2\beta-1)}{5\beta-3}>3$ is trivial. Combining  with \eqref{3.35}, we conclude $2 \leq p \leq 3$.
  \item For $\frac{9}{14}<\alpha<\frac{3}{4}$ and $\frac{1}{2}\leq\beta\leq\frac{3}{5}$, $\frac{3(1-\frac{2}{p})(6\alpha-3)}{4\alpha}-1 \leq 0$ implies $p \leq \frac{18(2\alpha-1)}{14\alpha-9}$.
   And  we also have $\frac{18(2\alpha-1)}{14\alpha-9}\geq3$ with $\frac{9}{14}<\alpha<\frac{3}{4}$.   Combining  with \eqref{3.35},  we get $2 \leq p \leq 3$.
  \item For $\frac{9}{14}<\alpha<\frac{3}{4}$ and $\beta>\frac{3}{5}$, we obtain $p\leq\min\{\frac{18(2\alpha-1)}{14\alpha-9}, \frac{6(2\beta-1)}{5\beta-3}\}$. Following the aforementioned procedure, we have $2 \leq p \leq 3$.
\end{enumerate}
Therefore,  the constraint conditions \eqref{3.35} and \eqref{3.36} can be simplified to  $2 \leq p \leq 3$   provided $\frac{1}{2}\leq\alpha<\frac{3}{4}$ and $\frac{1}{2}\leq\beta<1$.

 Considering $u\in\dot{H}^\alpha(\mathbb{R}^3),~b\in\dot{H}^\beta(\mathbb{R}^3)$ and Eq. \eqref{3.19}, when $R\rightarrow\infty$, we can obtain
$$2\bm{C}_{\alpha}\int_{\mathbb{R}^{4}_+}y^{\lambda_\alpha}|\bar{\nabla}u^*|^2~\text{d}x\text{d}y + 2\bm{C}_{\beta}\int_{\mathbb{R}^{4}_+}y^{\lambda_\beta}|\bar{\nabla}b^*|^2~\text{d}x\text{d}y\leq0,$$
which implies $\bm{C}_{\alpha}\int_{\mathbb{R}^{4}_+}y^{\lambda_\alpha}|\bar{\nabla}u^*|^2~\text{d}x\text{d}y=0$ and $\bm{C}_{\beta}\int_{\mathbb{R}^{4}_+}y^{\lambda_\beta}|\bar{\nabla}b^*|^2~\text{d}x\text{d}y=0$.
By Lemma \ref{le2.3} and $\eqref{Eq2.1}_4$, since $u$ is a smooth function in $L^p(\mathbb{R}^3)$, it is obvious that $u=b=0$.

In summary, combining Subsection \ref{subsubsec3.1.1} and Subsection \ref{subsec3.2}, we have $u=b=0$ for $\frac{1}{2}\leq\alpha,~\beta<1$ with $2\leq p\leq\frac{9}{2}$ and $1\leq q\leq\frac{3}{2}$.

On the other hand, together with $u=b=0$, equations \eqref{eq1.2} and $\eqref{eq1.1}_1$ imply that $\rho$ is constant on $\mathbb{R}^3$, completing the proof of Theorem \ref{th1.1}.
	

\section*{Author Contributions}
Zhenyuan Liu: 
Writing - original draft (equal).\\  Weihua Wang: Formal analysis; Writing - original draft (equal); Writing - review \& editing.

\section*{Acknowledgments}
This work was supported by National Natural Science Foundation of China (Grant No. 12271470). 

\section*{Funding}
This work was supported by  National Natural Science Foundation of China (Grant No. 12271470).
\section*{Conflicts of Interest}
The authors declare no conflicts of interest.

\section*{Data Availability Statement}
No data was used in this paper.
\section*{Declarations}

\noindent{\bf Consent to Publish declaration}: not applicable. This manuscript does not contain any individual person's data in any form.\\

\noindent{\bf Consent to Participate declaration}: not applicable. This research did not involve human participants or animals.\\

\noindent{\bf Ethics declaration }  not applicable.\\

\end{document}